\documentclass[12 pt]{amsart}
\usepackage{amssymb,latexsym,amsmath}
\topskip  \numberwithin{equation}{section}
\newtheorem{theorem}{Theorem}[section]

\newtheorem{conjecture}[theorem]{Conjecture}
\newtheorem{remark}[theorem]{Remark}

\newtheorem{problem}[theorem]{Problem}

\begin{document}

\pagenumbering{arabic}
\pagestyle{headings}

\title{On a conjecture of a logarithmically completely monotonic function}
\maketitle
\begin{center}
Valmir Krasniqi$^a$ and Armend Sh. Shabani$^a$

$^a$Department of Mathematics, University of Prishtina\\ Avenue
"Mother Theresa", 5 Prishtine 10000, Republic of Kosova

{\tt vali.99@hotmail.com, armend\_shabani@hotmail.com}

\end{center}

\section*{Abstract}

In this short note we prove a conjecture for the interval $(0,1)$, related to a logarithmically completely monotonic function, presented in \cite{BG}. Then, we extend by proving a more generalized theorem. At the end we pose an open problem on a logarithmically completely monotonic function involving $q$-Digamma function.

{\bf Key words}: completely monotonic, logarithmically completely monotonic

\noindent {\bf 2000 Mathematics Subject Classification}: 33D05,
26D07
\section{Introduction}
Recall from~\cite[Chapter~XIII]{mpf-1993}, \cite[Chapter~1]{Schilling-Song-Vondracek-2010} and \cite[Chapter~IV]{widder} that a function $f$ is said to be completely monotonic on an interval $I$ if $f$ has derivatives of all orders on $I$
and satisfies
\begin{equation}\label{CM-dfn}
0\le(-1)^{n}f^{(n)}(x)<\infty
\end{equation}
for $x\in I$ and $n\ge0$.
The celebrated Bernstein-Widder's Theorem (see \cite[p.~3, Theorem~1.4]{Schilling-Song-Vondracek-2010} or \cite[p.~161, Theorem~12b]{widder}) characterizes that a necessary and sufficient condition that $f$ should be completely monotonic for $0<x<\infty$ is that
\begin{equation} \label{berstein-1}
f(x)=\int_0^\infty e^{-xt}d\alpha(t),
\end{equation}
where $\alpha(t)$ is non-decreasing and the integral converges for $0<x<\infty$. This expresses that a completely monotonic function $f$ on $[0,\infty)$ is a Laplace transform of the measure $\alpha$.
\par
It is common knowledge that the classical Euler's gamma function $\Gamma(x)$ may be defined for $x>0$ by
\begin{equation}\label{gamma}
\Gamma(x)=\int_0^\infty t^{x-1}e^{-t} dt.
\end{equation}
The logarithmic derivative of $\Gamma(x)$, denoted by $\psi(x)=\frac{\Gamma'(x)}{\Gamma(x)}$, is called psi function or digamma function.
\par
An alternative definition of the gamma function $\Gamma(x)$ is
\begin{equation}\label{eqgammap2}
\Gamma(x) =\lim_{p\rightarrow\infty}\Gamma_p(x),
\end{equation}
where
\begin{equation}\label{eqgammap1}
\Gamma_p(x)=\frac{p!p^x}{x(x+1)\dotsm(x+p)}=\frac{p^x}{x(1+{x}/{1})\dotsm(1+{x}/{p})}
\end{equation}
for $x>0$ and $p\in\mathbb{N}$. See~\cite[p.~250]{Apostol-Number}. The $p$-analogue of the psi function $\psi(x)$ is defined as the
logarithmic derivative of the $\Gamma_p$ function, that is,
\begin{equation}\label{psi_p1}
\psi_p(x)=\frac{d}{ dx} \ln \Gamma_p(x)=\frac{\Gamma'_p(x)}{\Gamma_p(x)}.
\end{equation}
The function $\psi_p$ has the following properties (see~\cite[p.~374, Lemma~5]{MC-Krashiqi-Mansour-Shabani} and \cite[p.~29, Lemma~2.3]{VA}).
\begin{enumerate}
\item
It has the following representations
\begin{equation}\label{psi_series1}
\psi_p(x)=\ln p-\sum_{k=0}^{p}\frac{1}{x+k} =\ln p -\int_{0}^{\infty}\frac{1-e^{-(p+1)t}}{1-e^{-t}}e^{-xt}dt.
\end{equation}
\item
It is increasing on $(0,\infty)$ and $\psi'_p$ is completely monotonic on $(0,\infty)$.
\end{enumerate}

\par
In~\cite[pp.~374\nobreakdash--375, Theorem~1]{psi-alzer}, it was proved that the function
\begin{equation}\label{theta-alpha}
\theta_\alpha(x)=x^\alpha[\ln x-\psi(x)]
\end{equation}
is completely monotonic on $(0,\infty)$ if and only if $\alpha\le1$.

For the history, backgrounds, applications and alternative proofs of this conclusion, please refer to~\cite{theta-new-proof.tex-BKMS}, \cite[p.~8, Section~1.6.6]{bounds-two-gammas.tex} and closely-related references therein.
\par

A positive function $f$ is said to be {\em logarithmically completely monotonic} \cite{MC-Krashiqi-Mansour-Shabani}  on an open interval $I$, if $f$ satisfies
\begin{equation}\label{eq01}
(-1)^n[\ln f(x)]^{(n)}\geq 0, (x\in I, n=1,2,\ldots ).
\end{equation}
If the inequality (1.2) is strict, then $f$ is said to be {\em strictly logarithmically completely monotonic}.
Let $C$ and $L$ denote the set of completely monotonic functions and the set of logarithmically completely monotonic functions, respectively. The relationship between completely monotonic functions and logarithmically completely monotonic functions can be presented \cite{MC-Krashiqi-Mansour-Shabani} by $L\subset C$.

\section{Main results}

In \cite{BG} has been posed the following conjecture.

\begin{conjecture}
The function

\begin{equation}
q(t):=t^{t(\psi(t)-\log t)-\gamma}
\end{equation}

is logarithmically completely monotonic on $(0,\infty)$.
\end{conjecture}

\begin{theorem}
The function

\begin{equation}
q(t):=t^{t(\psi(t)-\log t)-\gamma}
\end{equation}

is logarithmically completely monotonic on $(0,1)$.
\end{theorem}

\begin{proof}
One easily finds that
\begin{equation}
\log q(t)=-t\cdot(\log t-\psi(t)) \log t - \gamma \cdot\log t
\end{equation}
Let $\displaystyle h(t)=-\gamma \cdot\log t, g(t)=-\log t; f(t)=t\cdot(\log t-\psi(t))$.
Alzer \cite{psi-alzer} proved that the function $f(t)=t\cdot(\log t-\psi(t))$ is strictly completely monotonic on $(0,\infty)$. The functions $g(t)=-\log t$ and $h(t)=-\gamma \cdot\log t$ are also strictly completely monotonic on $(0,1)$. We complete the proof by recalling the results from \cite{widder}.

1) The product of two completely monotone functions is completely monotonic function.

2) A non-negative finite linear combination of completely monotone functions
is completely monotonic function.
\end{proof}

We extend the previous result to the following theorem.
\begin{theorem}
The function
\begin{equation}
q_p(t):=t^{t\cdot\big(\psi_p(t)-\log \frac{pt}{t+p+1}\big)-\gamma}
\end{equation}
is logarithmically completely monotonic on $(0,1)$.
\end{theorem}

\begin{proof}
One easily finds that
\begin{equation}
\log q_p(t)=-t\big(\log \frac{pt}{t+p+1}-\psi(t)\big) \log t - \gamma \cdot\log t
\end{equation}

Let $\displaystyle h(t)=-\gamma \cdot\log t, g(t)=-\log t; f_p(t)=t\cdot\big(\log \frac{pt}{t+p+1}-\psi_p(t)\big)$.

Krasniqi and Qi  \cite{VQ} proved that the function $f_p(t)=t\cdot \big(\log \frac{pt}{t+p+1}-\psi_p(t)\big)$ is strictly completely monotonic on $(0,\infty)$. The functions $g(t)=-\log t$ and $h(t)=-\gamma \cdot\log t$ are also strictly completely monotonic on $(0,1)$.

By refering the same results from \cite{widder} as in previous proof, we complete the proof.
\end{proof}

\begin{remark}
Letting $p\to\infty$ in Theorem 2.3 , we obtain Theorem 2.2 .
\end{remark}
At the end we pose the following open problem:

\begin{problem}
Let $\psi_q(t)$ be $q$-Digamma function. Find the family of functions $\theta(t)$  such that
\begin{equation}
\displaystyle q(t):=t^{t\cdot(\psi_q(t)-\log \theta(t))-\gamma}
\end{equation}
is logarithmically completely monotonic on $(0,\infty)$.
\end{problem}

\begin{remark}
This is a corrected version of paper \cite{valimendi}
\end{remark}


\end{document}